\documentclass[12pt,a4paper, oneside, bold,secthm,seceqn,amsthm,ussrhead,reqno]{article}
\usepackage[utf8]{inputenc}
\usepackage[english]{babel}
\usepackage[symbol]{footmisc}
\usepackage{amssymb,amsmath,amsthm,amsfonts,xcolor,enumerate,hyperref, comment,longtable, cleveref}
\usepackage{verbatim}
\usepackage{times}
\usepackage{cite}
\usepackage{pdflscape}
\usepackage{ulem}
\usepackage[mathcal]{euscript}
\usepackage{tikz}
\usepackage{hyperref}
\usepackage{cancel}
\usepackage{stmaryrd}
 
\usepackage{amsfonts}
\usepackage{amssymb}
\usepackage{times}
\usepackage{xcolor}

\usepackage{float}
\usepackage{tasks}
\usepackage{array}

\usepackage{cite}
\usepackage{hyperref}

 \usepackage{fancyhdr} 
\usepackage{amsfonts}
\usepackage{amsmath}
\usepackage{eurosym}
\usepackage{geometry}

\usepackage{caption,booktabs}

\theoremstyle{plain}
\newtheorem{theorem}{Theorem}
\newtheorem{lemma}[theorem]{Lemma}
\newtheorem{proposition}[theorem]{Proposition}

\newtheorem{corollary}[theorem]{Corollary}
\newtheorem*{conjecture}{Conjecture}

\usetikzlibrary{arrows}

\usepackage{fouriernc}

\begin{document}

\noindent{\Large
Mutations of $\mathfrak{perm}$   algebras}\footnote{
The  first part of the work was supported by the Science Committee of the Ministry of Education and Science of the Republic of Kazakhstan (Grant No. AP22683764); 
FCT  UIDB/MAT/00212/2020  and UIDP/MAT/00212/2020.
The second part of the work was supported by the Russian Science Foundation under grant 22-71-10001.
}\footnote{We are grateful to the anonymous reviewer for the careful reading of our text and valuable comments. }\footnote{Corresponding author: Farukh Mashurov   (farukh.mashurov@sdu.edu.kz)}

 \bigskip

\begin{center}

 {\bf
Ivan Kaygorodov\footnote{CMA-UBI, University of  Beira Interior, Covilh\~{a}, Portugal; \ Moscow Center for Fundamental and Applied Mathematics,      Russia; 
  \    Saint Petersburg  University, Russia;\       kaygorodov.ivan@gmail.com}
  \&
Farukh Mashurov\footnote{ SDU University, Kaskelen, Kazakhstan; \ 	farukh.mashurov@sdu.edu.kz}

}

\end{center} 

 \medskip 
 
\noindent{\bf Abstract}:
{\it  
We describe mutation elements in free $\mathfrak{perm}$ algebras. 
Moreover, we construct a base of free mutation of free $\mathfrak{perm}$ algebra. Using Cohn's criterion for the speciality of algebras, we show that there is an exceptional homomorphic image of mutation of free $\mathfrak{perm}$ algebras.

}

 \medskip


 \bigskip
 

 \medskip

\section{ Introduction}

The notion of quasi-associative algebras first appeared in one seminal paper by Albert \cite{albert}. 
Namely, an associative algebra with a new multiplication $ x\star y=\lambda xy +(1-\lambda)yx$ for a fixed element $\lambda$ from the basic field gives a quasi-associative algebra.
If $\lambda=\frac{1}{2}$ we have the case of special Jordan algebras.
 Quasi-associative algebras play the principal role in the classification of simple non-commutative Jordan algebras \cite{oe}.
On the other hand, the same type of the modification of the original algebra product can be considered in non-associative cases too \cite{Ded}.
Another related notion ($q$-algebras) was introduced by Dzhumadil'daev in \cite{dzh08,dzh09}. 
He considered a new algebra product $ x\star y= xy +q yx$ on Leibniz and Zinbiel algebras.
A version of $q$-algebras for algebraic structures with two multiplications is given in \cite{bai23} for $q$-dendriform algebras. 
It is well-known that $(-1)$-associative algebra is a Lie algebra.
As it is known now, the consideration of $q$-algebras gives new interesting types of algebras.
So, 
the class of $(-1)$-Zinbiel algebras gives Tortkara algebras \cite{dzh09},
the class of $1$-Novikov algebras gives Tortken algebras \cite{dzh02}
and 
the class of $2$-Novikov algebras gives admissible Novikov algebras \cite{bai22}.
Let us also mention that 
$1$-alternative and  $1$-noncommutative Jordan algebras are Jordan algebras;
$(-1)$-bicommutative, 
$(-1)$-Novikov, 
$(-1)$-assosymmetric,
$(-1)$-Pre-Lie algebras are Lie algebras;
$(-1)$-alternative  algebras are Malcev algebras; 
$(-1)$-assocyclic algebras are binary Lie algebras; 
and $(-1)$-$\mathfrak{perm}$ algebras are metabelian Lie algebras \cite{FM-BS}.

 The considered constructions admit a good generalization named mutation \cite{eld91,book mutation}. 
Let us fix now two elements $p$ and $q$  from the underline vector space of our algebra ${\rm A}$ and 
consider the new multiplication 
\begin{center}
$\langle x,y\rangle=(xp)y-(yq)x.$
\end{center}
The new algebra ${\rm A}_{p,q}$ with multiplication  $\langle \cdot, \cdot\rangle$ is called $(p,q)$-mutation of ${\rm A}$
and the multiplication  $\langle \cdot,\cdot\rangle$ is called the \textit{mutation} product.
The case of $q=0$ gives a left homotope \cite{mcc}, which in the case of left commutative or left symmetric algebras gives the notion of the Kantor square  \cite{FK21}.
Mutation of an algebra is a useful tool for constructing new algebras \cite{pump}.
 
 There is one identity in the case of mutations of associative algebras in degree three
\begin{equation}\label{liadm ident}
\sum_{\sigma\in \mathbb  S_3}sgn(\sigma) \langle\langle x_{\sigma(1)},x_{\sigma(2)}\rangle, x_{\sigma(3)}\rangle=
\sum_{\sigma\in \mathbb  S_3}sgn(\sigma)\langle x_{\sigma(1)},\langle x_{\sigma(2)},x_{\sigma(3)}\rangle\rangle. \end{equation}
  
  Montaner described the identities of degree less than five in mutations of associative algebras \cite{Montaner}. 
  It is an open question to describe all identities in mutations of associative algebras.  Recently, the results of  Montaner improved up to degree six by   Bremner, Brox, and Sánchez-Ortega in \cite{bremner-Juana}.

  \medskip 
  
An associative algebra with the left-commutative identity
\begin{equation*}\label{perm id}
abc=bac \end{equation*}
is called a $\mathfrak{perm}$   algebra (see \cite{Perm 1}). 
The free $\mathfrak{perm}$ algebra plays an important role in operad theory \cite{bremner-dotsenko, Vallette}.
In \cite{FM-BS}, it was shown that a $\mathfrak{perm}$ algebra under commutator is a metabelian Lie algebra and described Lie and Jordan elements in a free $\mathfrak{perm}$ algebra.
The necessary and sufficient conditions for a left-symmetric algebra to be embeddable into a differential $\mathfrak{perm}$ 
 algebra were studied in  \cite{KS22}.
Relations between Novikov dialgebras and $\mathfrak{perm}$ algebras were studied in \cite{MS22}.

Let $\mathfrak{Perm}$ be a variety of $\mathfrak{perm}$ algebras. Let $P\in \mathfrak{Perm}$ and ${ P}_{p,q}= ({ P},\langle\cdot,\cdot\rangle)$ be a $(p,q)$-mutation of ${P}$ under the mutation product, i.e.  multiplication  $\langle \cdot,\cdot\rangle.$ Define $\mathfrak{Perm}_{p,q}$ as the class of all mutations of all $\mathfrak{perm}$ algebras. Any algebra $\rm B$ is called \textit{special} concerning the multiplication  $\langle \cdot,\cdot \rangle,$ if there exist a $\mathfrak{perm}$ algebra $P$ and $p,q\in P$ such that $\rm  B$ is a subalgebra of $P_{p,q}$, otherwise it is called \textit{exceptional}.
Let $X=\{x_1,x_2,\ldots\}$ be a set and $P(X)$ be the free $\mathfrak{perm}$ algebra generated by $X$. A polynomial in $P(X)$ is called \textit{mutation} element of $P(X)$ if it can be expressed by elements of $X$ in terms of multiplication $\langle \cdot,\cdot\rangle.$

We consider all algebras over a field $\mathbb F$ of characteristic $0$.
Let us fix the following notations:
$$\begin{array}{rcl}
[x,y] & := & xy-yx,\\
\langle x, y\rangle & := & (xp)y-(yq)x,\\
\langle x, y, z  \rangle  & := & \langle \langle x, y\rangle, z  \rangle - \langle x, \langle y, z \rangle \rangle.
\end{array}$$ 

In this paper, we describe mutation elements in free $\mathfrak{perm}$ algebras. 
Moreover, we construct a base of free mutation of free $\mathfrak{perm}$ algebras. 
The construction of bases of free objects is not new 
(the first problems of this type were solved in the 1950s), 
but it is still an actual problem (see, for example, a recent paper \cite{B} and references therein).
Using Cohn's criterion for the specialty of algebras \cite{Cohn}, we show that there is an exceptional homomorphic image of mutation of free $\mathfrak{perm}$ algebras. We also give a necessary and sufficient condition for algebras whose mutation algebras are Lie-admissible.


\section{$(p,q)$-mutations of $\mathfrak{perm}$ algebras}

Let $P(X)$ be the free $\mathfrak{perm}$ algebra generated by $X$ over $\mathbb F,$ where   
$X=\{x_1,x_2,\ldots\}$ is a set of generators. 
For $a_1, \ldots, a_n\in P(X)$ denote by $a_1a_2\ldots a_{n-1}a_n$ 
an element $((\ldots(a_1a_2) \ldots )a_{n-1})a_n.$
For two elements $p,q\notin P(X),$ we consider the $(p,q)$-mutation $P_{p,q}(X\cup\{p,q\})$ of the free $\mathfrak{perm}$ algebra $P(X\cup\{p,q\})$ with  the set of generators $X\cup\{p,q\}.$ 
Let $P_{p,q}(X)$ be a subalgebra of $P_{p,q}(X\cup\{p,q\})$ generated by the set $X.$
We define the set $B=X\cup B_1\cup B_2\cup B_3$ where  
\begin{longtable}{lcl}
$B_1$&$=$&$\Big\{x_ipx_j-x_j qx_i\in P(X\cup\{p,q\})\Big|x_i,x_j\in X\Big\},$\\ 

$B_2$&$=$&$\Big\{(p-q)^{n-1}x_{j_n}\ldots x_{j_1}\in P(X\cup\{p,q\}) \Big| \begin{array}{l}x_{j_1},\ldots,x_{j_n}\in X, n>2 \\ 
j_2\leq j_3\leq\ldots\leq j_n \end{array} \Big\},$ \\

$B_3$&$=$&$ \Big\{p^{n-1-i} q^i x_{j_n}\ldots x_{j_3}[x_{j_2},x_{j_1}]\in P(X\cup\{p,q\})
\Big| \begin{array}{l}x_{j_1},\ldots,x_{j_n}\in X,  1\leq i\leq n-1,
\\
 
j_2>j_1\leq j_3\leq\ldots\leq j_n,n>2  \end{array} \Big\}.$
\end{longtable}
In \cite{FM-BS} it was proved that $\mathfrak{perm}$ algebra under commutator is a metabelian Lie algebra. Since, $$x_{j_n}\ldots x_{j_3}[x_{j_2},x_{j_1}] \ = \ 
[x_{j_n},[\ldots ,[x_{j_3},[x_{j_2},x_{j_1}]]\ldots]],$$
 we have every element of the form $x_{i_n}\ldots x_{i_3}[x_{i_2},x_{i_1}]$ is a linear combination of elements  \begin{equation*}\label{base of Meta.Lie}
       x_{j_n}\ldots x_{j_3}[x_{j_2},x_{j_1}],
    \end{equation*}
	where $j_2>j_1\leq j_3\leq\ldots\leq j_n.$ Therefore, throughout the article, we will not dwell on the details of rewriting elements of this type.

\subsection{Mutation elements in a free $\mathfrak{perm}$ algebra}

In this subsection, we describe mutation elements in $P(X\cup\{p,q\}).$  

To simplify the expressions, we note the following obvious identities in every $\mathfrak{perm}$ algebra:
\begin{equation}
\label{relations} \begin{array}{c}
\langle a, b \rangle \ =\  (p - q)ab + q [a, b],
\\
\langle a, bc \rangle \ =\  b \langle a, c \rangle,
\\
\langle ab, c \rangle \ =\ a \langle b, c \rangle,
\\
\langle a, [b,c] \rangle \ =\  p a[b,c],
\\
\langle [a,b], c \rangle \ = \ -q c[a,b]. \\
\end{array}  
\end{equation}

Although these identities are straightforward, they are useful for rewriting elements in the desired form. For example, we can show that the following identity holds in a $\mathfrak{perm}$ algebra using the identities above:
\begin{equation}\label{hom ident1}
     \langle b,\langle a, c\rangle \rangle \ =\ a p \langle b, c\rangle - c q\langle b, a\rangle.
\end{equation}

Moreover, the following proposition follows from direct calculations using the identities \eqref{relations}.

\begin{proposition}\label{open brackets deg3}
Let $  x_1 ,  x_2 ,  x_3  \in P_{p,q}( X),$ then 

\begin{longtable}{rcl}
$\langle\langle   x_1 ,  x_2 \rangle,  x_3 \rangle$&$=$&$(p-q)^2  x_1   x_2   x_3 +pq  x_1 [  x_2 ,  x_3 ]-q^2   x_2 [  x_1 ,  x_3 ],$\\

$\langle   x_1 ,\langle   x_2 ,  x_3 \rangle\rangle$&$=$&$(p-q)^2  x_1   x_2   x_3 +pq  x_1 [  x_2 ,  x_3 ]+pq  x_2 [  x_1 ,  x_3 ]-q^2   x_2 [  x_1 ,  x_3 ],$\\

$\langle \langle   x_1 ,   x_2 \rangle ,\langle   x_3 ,  x_4 \rangle\rangle $&$ = $&$ (p-q)^3   x_1    x_2    x_3    x_4 
+(p-q)pq   x_1    x_2  [  x_3 ,   x_4 ]+$ \\
&&\multicolumn{1}{r}{$  (p-q)pq   x_1    x_3  [  x_2 ,   x_4 ]
-(p-q)q^2   x_2    x_3  [  x_1 ,   x_4 ].
$}
\end{longtable} 
\end{proposition}



In particular, the following two lemmas prove that the elements of the sets $B_2$ and $B_3$ are mutation elements in  $P(X\cup\{p,q\}).$ We write $a\equiv b$ if $a-b\in P_{p,q}(X).$

\begin{lemma}\label{mutation element1 lemma}
Let $  x_1 ,  x_2 , \ldots,   x_n  \in P_{p,q}( X),$ then 
\begin{equation}\label{mutation element1}
    (p-q)^{n-1}  x_n \ldots   x_1 \equiv0,\,\, n>2.
\end{equation}\end{lemma}

\begin{proof} 
We prove it by induction on $n.$ The base of induction is $n = 3$. 
Using Proposition $\ref{open brackets deg3}$ we have
\begin{center}
$(p-q)^2  x_3   x_2   x_1 =-\langle\langle   x_3 ,  x_1 \rangle,  x_2 \rangle+2\langle\langle   x_3 ,  x_2 \rangle,  x_1 \rangle+\langle   x_1 ,\langle   x_3 ,  x_2 \rangle\rangle-\langle   x_3 ,\langle   x_2 ,  x_1 \rangle\rangle.$\end{center}

Suppose that the equality $(\ref{mutation element1})$ holds for $n-1$. Then we have
	 $$(p-q)^{n-2}  x_{n-1} \ldots   x_1 \equiv0.$$
	 
	 If we put  $\langle   x_n ,  x_{n-1} \rangle$ instead of $  x_{n-1} ,$ we obtain  
 \begin{center}
 $(p-q)^{n-2}\langle   x_n ,  x_{n-1} \rangle   x_{n-2}\ldots   x_1 =(p-q)^{n-2}(  x_n  p   x_{n-1} -  x_{n-1} q  x_n )   x_{n-2}\ldots   x_1 \equiv0.$
 \end{center}
 
 By left-commutative identity, we have
 \begin{center} $(p-q)^{n-2}(p-q)  x_{n}  x_{n-1} \ldots   x_1 =(p-q)^{n-1}  x_{n}  x_{n-1} \ldots   x_1 \equiv0.$
 \end{center}
\end{proof}

\begin{lemma}\label{mutation element2 lemma}
Let $  x_1 ,  x_2 , \ldots,   x_n  \in P_{p,q}( X ),$ then 
for $i\in\{1,\ldots,n-1\}$ and $n>2$ we have
\begin{equation}\label{mutation element2}
    p^{n-1-i} q^i   x_n \ldots   x_3 [  x_2 ,  x_1 ]\equiv0,
\end{equation} 
\end{lemma}
\begin{proof} 
We prove it by induction on $n.$ The base of induction is $n = 3$. Using Proposition $\ref{open brackets deg3}$ we have 
\begin{longtable}{llclcl}
for $i=1:$ \ & $\langle\langle   x_1 ,  x_3 \rangle,  x_2 \rangle-\langle   x_1 ,\langle   x_3 ,  x_2 \rangle\rangle$&$ =$&$ pq   x_3 [  x_2 ,  x_1 ]$&$  \equiv $&$   0,$\\

for $i=2:$ \ & $\langle   x_1 ,\langle   x_2 ,  x_3 \rangle\rangle-\langle   x_2 ,\langle   x_1 ,  x_3 \rangle\rangle$&$ =$&$ q^2   x_3 [  x_2 ,  x_1 ]$&$  \equiv $&$ 0.$
\end{longtable}
Suppose that the equality $(\ref{mutation element2})$ holds for $n-1$.
Then   for $i\in\{1,\ldots,n-2\}$ we have
\begin{equation}\label{assumption mutation element2} p^{n-2-i} q^i   x_{n-1}   x_{n-2}\ldots   x_3 [  x_2 ,  x_1 ]\equiv0.\end{equation}
	
	Set $z=  x_{n-2}\ldots   x_3 $ and put  $\langle   x_n ,  x_{n-1} \rangle$ instead of $  x_{n-1} $ in  $(\ref{assumption mutation element2}).$  Then by left-commutative identity we have 
	  \begin{equation}\label{pq1}
	      p^{n-2-i} q^i \langle   x_n ,  x_{n-1} \rangle z[  x_2 ,  x_1 ]=   p^{n-2-i} q^i(p-q)   x_n    x_{n-1}  z[  x_2 ,  x_1 ]\equiv0.
	  \end{equation}
 
On the other hand, if we put  $\langle   x_2 ,  x_{n}\rangle$ instead of $  x_{2}$ in  $(\ref{assumption mutation element2})$   we have 
\begin{flushleft}$ p^{n-2-i} q^i    x_{n-1} z[\langle   x_2 ,  x_{n}\rangle,  x_1 ]= p^{n-2-i} q^i    x_{n-1} z[   x_2 p  x_{n}-  x_{n}q  x_2 ,  x_1 ]\ =$\end{flushleft}
\begin{equation}\label{pq2}
   = \  p^{n-2-i} q^i (p-q)   x_{n}  x_{n-1} z   x_2   x_1 -p^{n-2-i} q^i p   x_{n-1} z   x_2   x_1   x_{n}+p^{n-2-i} q^i q   x_{n}  x_{n-1} z   x_1   x_2  \equiv 0.
    \end{equation}

Consider the difference of the expressions $(\ref{pq1})$ and $(\ref{pq2}),$ then 

\begin{flushleft}$p^{n-2-i} q^i(p-q)   x_n    x_{n-1}  z   x_2    x_1 - p^{n-2-i} q^i(p-q)   x_n    x_{n-1}  z   x_1    x_2 \ -$
\end{flushleft}
\begin{center}
$\big(p^{n-2-i} q^i (p-q)   x_{n}  x_{n-1} z   x_2   x_1 -
p^{n-2-i} q^i p   x_{n-1} z   x_2   x_1   x_{n}+
p^{n-2-i} q^i q   x_{n}  x_{n-1} z   x_1   x_2 \big) \  =$
\end{center}
\begin{flushright}$= \ p^{n-2-i} q^i p   x_{n-1} z   x_1 [  x_2 ,  x_{n}]\equiv 0.$
\end{flushright}

Therefore, we obtain for $i\in\{1,2, \ldots, n-2\}$ 

 \begin{equation}\label{rel mut elem2}
     p^{n-1-i} q^i   x_n \ldots   x_3 [  x_2 ,  x_1 ]\equiv0.\end{equation}
Now it remains to show that for $i=n-1$ and $n>3,$ the following is true 
$q^{n-1}   x_n \ldots   x_3 [  x_2 ,  x_1 ]\equiv0.$
By Lemma $\ref{mutation element1 lemma},$ it is easy to see that for $u \in P_{p,q}(X)$ we have
$(p-q)^2   x_n    x_{n-1}  u\equiv0.$
By induction on $n$ and $(\ref{rel mut elem2})$ if we set $u=q^{n-3}   x_{n-2}\ldots [  x_2 ,  x_1 ]\in P_{p,q}(X)$ we obtain
\begin{center}
$(p-q)^2 q^{n-3}   x_n    x_{n-1}     x_{n-2} \ldots [  x_2 ,  x_1 ]\equiv0.$
\end{center}
By $(\ref{rel mut elem2})$ we have  $$q^{n-1}   x_n \ldots   x_3 [  x_2 ,  x_1 ]\equiv0.$$
This completes the proof.
 
\end{proof}

\subsection{Mutation product of elements of the set $B$}
Our aim in this subsection is to construct a linear basis for $P_{p,q}( X ).$ 

For every monomial $v \in P_{p,q}(X),$ we define the degree of $v$ as the number of $x_i$'s  in $v$.
First, we show that $P_{p,q}( X )$ is spanned by the set $B.$

\begin{lemma}\label{mut product of B} Every element of $P_{p,q}(X)$ is a linear combination of elements of the set $B.$
\end{lemma}

\begin{proof} We prove the statement by induction on degree $n.$ Let $v=\langle w_l,u_m\rangle$ has degree $n,$ where $ w_l,u_m\in P_{p,q}(X)$ whose degrees are $l$ and $m,$ respectively, and $l+m=n.$
The base of induction when $n=2$ is evident, for $n=3$  follows from Proposition \ref{open brackets deg3}.  Suppose that the statement is true if the degree of $v\in P_{p,q}(X)$ is less than $n>3.$  By induction on $n$ we may assume that $ w_l$ and $u_m$ are in $\mathrm{span}(B)$. We want to show that $$v=\langle w_l,u_m\rangle\in \mathrm{span}(B).$$

Let $l=1$ then $w_1\in X.$ We have $u_m\in B_2\cup B_3.$
\begin{itemize}
    \item If $u_m=(p-q)^{m-1}x_{j_m}\ldots x_{j_1}\in B_2,$ then by \eqref{relations} we have 
    \begin{flushleft}
    $\langle   w_1,(p-q)^{m-1}x_{j_m}\ldots x_{j_1} \rangle\  =\ (p-q)^{m-1}x_{j_m}\ldots x_{j_{2}} \langle   w_1, x_{j_1} \rangle \ =$\end{flushleft}
 \begin{flushright}$= \ (p-q)^{m}   w_1   x_{j_m}\ldots x_{j_1} +(p-q)^{m-1}q  x_{j_m}\ldots  [  w_1,  x_{j_1} ]\ \in \ \mathrm{span}(B).$\end{flushright}

  \item If $u_m=p^{m-1-i} q^i x_{j_m}\ldots [x_{j_{2}},x_{j_1}]\in B_3,$ then by \eqref{relations} we have 

 \begin{flushleft} $\langle   w_1, p^{m-1-i} q^i x_{j_m}\ldots [x_{j_{2}},x_{j_1}]\rangle \ =\  p^{m-1-i} q^i x_{j_m}\ldots x_{j_3} \langle   w_1,[x_{j_{2}},x_{j_1}]\rangle \ =$\end{flushleft}
 \begin{flushright}$= \ p^{m-i} q^{i}   w_1 x_{j_m}\ldots [x_{j_{2}},x_{j_1}]\ \in \ \mathrm{span}(B).$\end{flushright} 
\end{itemize}

Let $l=2$ then $w_2=\langle x_{i_1},x_{i_2}\rangle\in B_1.$ 

\begin{itemize}
    \item If $u_2\in B_1,$ then it follows from the Proposition \ref{open brackets deg3}. 
    \item If $u_m=(p-q)^{m-1}x_{j_m}\ldots x_{j_1}\in B_2,$ then by \eqref{relations} and Proposition \ref{open brackets deg3} we have 
  
\begin{flushleft}
$\langle  \langle  x_{i_1},x_{i_2} \rangle, (p-q)^{m-1}x_{j_m}\ldots x_{j_1} \rangle \ =\ (p-q)^{m-1}x_{j_m}\ldots x_{j_{2}}\langle  \langle  x_{i_1},x_{i_2} \rangle,  x_{j_1} \rangle\ =$\end{flushleft}
 \begin{center}
$= \ (p-q)^{m+1}  x_{i_1} x_{i_2}   x_{j_m} \ldots   x_{j_1} 
-(p-q)^{m-1}pq   x_{j_m} \ldots   x_{i_1}[  x_{j_1} ,   x_{i_2}]+$ \end{center}
 \begin{flushright}
$ (p-q)^{m-1}q^2   x_{j_m} \ldots    x_{i_2}[  x_{j_1} ,  x_{i_1}] 
\ \in  \ \mathrm{span}(B).$
 \end{flushright}

\item If $u_m=p^{m-1-i} q^i   x_{j_m}  \ldots [  x_{j_2} ,  x_{j_1} ]\in B_3,$ then by \eqref{relations} and Proposition \ref{open brackets deg3} we have 
\begin{flushleft}
$\langle \langle   x_{i_1}, x_{i_2}  \rangle,p^{m-1-i} q^i   x_{j_m}  \ldots [  x_{j_2} ,  x_{j_1} ]\rangle \ = \ p^{m-1-i} q^i   x_{j_m}  \ldots x_{j_{2}} \langle \langle   x_{i_1}, x_{i_2}  \rangle,[  x_{j_2} ,  x_{j_1} ]\rangle \ = $\end{flushleft}
 \begin{flushright}$= \ p^{m+1-i} q^i    x_{i_1} x_{i_2}  x_{j_m}  \ldots [  x_{j_2} ,  x_{j_1} ]-p^{m-i} q^{i+1}   x_{i_1} x_{i_2}  x_{j_m}  \ldots [  x_{j_2} ,  x_{j_1} ]\ \in \  \mathrm{span}(B).$ \end{flushright}
\end{itemize}

Assume that $l>2,$ then $w_l\in B_2\cup B_3.$  Let $w_l=(p-q)^{l-1}  x_{i_l} \ldots   x_{i_1}\in B_2.$

\begin{itemize}

    \item If $m=1,$ then $u_1\in X$ and by \eqref{relations} we have  
  \begin{flushleft}  $\langle   (p-q)^{l-1}  x_{i_l} \ldots   x_{i_1}, u_{1} \rangle \ =\ (p-q)^{l-1}  x_{i_l} \ldots   x_{i_2} \langle x_{i_1}  , u_{1} \rangle \ =$\end{flushleft}
    \begin{flushright}
    $=\ (p-q)^{l} x_{i_l}\ldots x_{i_1} u_1+(p-q)^{l-1} q x_{i_l}\ldots x_{i_2}   [x_{i_1}  , u_1]\ \in \ \mathrm{span}(B) .$
    \end{flushright}

    \item If $m=2$ then $u_2=\langle x_{j_1},x_{j_2}\rangle\in B_1,$ and by \eqref{relations} and Proposition \ref{open brackets deg3} we have 
        \begin{flushleft}
     $\langle(p-q)^{l-1}  x_{i_l} \ldots   x_{i_1} ,  \langle x_{j_1},x_{j_2}\rangle \rangle= (p-q)^{l-1}  x_{i_l} \ldots   x_{i_2}\langle x_{i_1} ,  \langle x_{j_1},x_{j_2}\rangle \rangle \ =$\end{flushleft}
    \begin{center}
    $(p-q)^{l+1}   x_{i_l} \ldots   x_{i_1} x_{j_1}x_{j_2}-(p-q)^{l-1} p q   x_{i_l} \ldots   x_{i_1}[x_{j_2},x_{j_1}]+$ \end{center}
      \begin{flushright}$(p-q)^{l-1}p q   x_{i_l} \ldots   x_{i_2} x_{j_1}[  x_{i_1} ,x_{j_2}]- (p-q)^{l-1}q^2   x_{i_l} \ldots   x_{i_2}  x_{j_2}[  x_{i_1} ,x_{j_1}] \ \in\  \mathrm{span}(B) .$\end{flushright}
    
    \item  If $u_m=(p-q)^{m-1}x_{j_m}\ldots x_{j_1}\in B_2,$ then by \eqref{relations} we have 
    \begin{flushleft}
    $\langle   w_l,u_m \rangle \ =\ (p-q)^{l+m-2} x_{i_l}\ldots x_{i_2}  x_{j_m}\ldots x_{j_{2}} \langle x_{i_1}  , x_{j_1} \rangle \ =$\end{flushleft}
    \begin{center}
    $=\ (p-q)^{l+m-1} x_{i_l}\ldots x_{i_2}  x_{j_m}\ldots x_{j_{2}} x_{i_1} x_{j_1}+(p-q)^{l+m-2} q x_{i_l}\ldots x_{i_2}  x_{j_m}\ldots x_{j_{2}} [x_{i_1}  , x_{j_1}] $\end{center}
        \begin{flushright}
$\in\  \mathrm{span}(B) .$
    \end{flushright}
    \item If $u_m=p^{m-1-i} q^i x_{j_m}\ldots [x_{j_{2}},x_{j_1}]\in B_3,$ then by \eqref{relations} we have 
 \begin{flushleft} $\langle   w_l,u_m \rangle =\langle   (p-q)^{l-1}  x_{i_l} \ldots   x_{i_1}, p^{m-1-i} q^i x_{j_m}\ldots [x_{j_{2}},x_{j_1}]\rangle \ = $\end{flushleft}
 \begin{center}$= \  p^{m-1-i} q^i (p-q)^{l-1}  x_{i_l} \ldots   x_{i_2} x_{j_m}\ldots x_{j_3} \langle   x_{i_1},[x_{j_{2}},x_{j_1}]\rangle \ =$\end{center} 
  \begin{flushright}$= \ p^{m-i} q^i (p-q)^{l-1}  x_{i_l} \ldots   x_{i_2} x_{i_1} x_{j_m}\ldots x_{j_3} [x_{j_{2}},x_{j_1}] \ \in \  \mathrm{span}(B).$\end{flushright} 

\end{itemize}

Finally, let $w_l=p^{l-1-i} q^i    x_{i_l} \ldots  [x_{i_2}, x_{i_1}]\in B_3.$

\begin{itemize}

    \item If $m=1,$ then $u_1\in X$ and by \eqref{relations} we have \begin{flushleft}
    $\langle   p^{l-1-i} q^i    x_{i_l} \ldots  [x_{i_2}, x_{i_1}], u_{1} \rangle =p^{l-1-i} q^i    x_{i_l} \ldots x_{i_3}  \langle [x_{i_2}, x_{i_1}]  , u_{1} \rangle =$\end{flushleft}
    \begin{flushright}
    $=\ -p^{l-1-i} q^{i+1}   u_{1} x_{i_l} \ldots  [x_{i_2}, x_{i_1}] \ \in\  \mathrm{span}(B) .$
    \end{flushright}

    \item If $m=2$ then $u_2=\langle x_{j_1},x_{j_2}\rangle\in B_1,$ and by \eqref{relations} and Proposition \ref{open brackets deg3} we have 
        \begin{flushleft}
     $\langle p^{l-1-i} q^i    x_{i_l} \ldots  [x_{i_2}, x_{i_1}] ,  \langle x_{j_1},x_{j_2}\rangle \rangle\ =\  p^{l-1-i} q^i    x_{i_l} \ldots  x_{i_3} \langle [x_{i_2}, x_{i_1}] ,  \langle x_{j_1},x_{j_2}\rangle \rangle \ =$\end{flushleft}
    \begin{center}
    $-p^{l-1-i} q^{i+1}  \langle x_{j_1},x_{j_2}\rangle  x_{i_l} \ldots  x_{i_3}  [x_{i_2}, x_{i_1}] \ =\ -p^{l-1-i} q^{i+1} (p-q) x_{j_1}x_{j_2} x_{i_l} \ldots  x_{i_3}  [x_{i_2},x_{i_1}] $\end{center}
        \begin{flushright}$\in \ \mathrm{span}(B) .$\end{flushright}
    
    \item  If $u_m=(p-q)^{m-1}x_{j_m}\ldots x_{j_1}\in B_2,$ then by \eqref{relations} we have 
    \begin{flushleft}
    $\langle   w_l,u_m \rangle \ =\ (p-q)^{m-1}  p^{l-1-i} q^i      x_{j_m}\ldots x_{j_{2}}  x_{i_l}\ldots x_{i_3}   \langle [x_{i_2}, x_{i_1}],x_{j_1}  \rangle \ =\ $\end{flushleft}
    \begin{flushright}
  $= \ (p-q)^{m-1}  p^{l-1-i} q^{i+1}      x_{j_m}\ldots x_{j_{1}}  x_{i_l}\ldots x_{i_3}    [x_{i_2}, x_{i_1}]  \ \in\  \mathrm{span}(B) .$
    \end{flushright}
    \item If $u_m=p^{m-1-i} q^i x_{j_m}\ldots [x_{j_{2}},x_{j_1}]\in B_3,$ then by \eqref{relations} we have 

 \begin{flushleft} $\langle   w_l,u_m \rangle\  =\ \langle p^{l-1-i} q^i    x_{i_l} \ldots  [x_{i_2}, x_{i_1}], p^{m-1-i} q^i x_{j_m}\ldots [x_{j_{2}},x_{j_1}]\rangle \ = \ $\end{flushleft}
 \begin{flushright}$ =\ p^{l+m-2-2i} q^{2i}   x_{i_l} \ldots   x_{i_3} x_{j_m}\ldots x_{j_3} \langle  [x_{i_2}, x_{i_1}],[x_{j_{2}},x_{j_1}]\rangle \ =\  0 \ \in \ \mathrm{span}(B).$\end{flushright} 

\end{itemize}
This completes the proof.
\end{proof}

Now, we are ready to prove the main result of this section. 

\begin{theorem}\label{mut element and dim} Let  $X=\{x_{1}, \ldots,x_{n}\},$ then the set $B$ forms a base of $P_{p,q}(X).$  Moreover, the multilinear part of $P_{p,q}(X)$ has dimension
 $n+(n-1)^2$ for $n>2.$
\end{theorem}

\begin{proof}
 It is clear that the elements of $B$ are linearly independent in $P_{p,q}(X).$ 
By Lemma $\ref{mut product of B}$ we have that the set $B$ forms a base of $P_{p,q}(X).$ 

\end{proof}


\section{Identities in    mutations of $\mathfrak{perm}$ algebras}

In this section, we consider the identities of mutations of  $\mathfrak{perm}$ algebras and describe all identities of degree three.

\begin{lemma}\label{lem standard identity 3}
Let $P$ be a $\mathfrak{perm}$ algebra. Then $P_{p,q}$ satisfies the identity  \begin{equation}\label{standard identity 3}
    f(x_1,x_2,x_3):=\sum_{\sigma\in \mathbb S_3}{sgn(\sigma)\langle\langle x_{\sigma(1)},x_{\sigma(2)}\rangle, x_{\sigma(3)}\rangle}=0.\end{equation}
\end{lemma}
\begin{proof} By Proposition $\ref{open brackets deg3}$ we have  

\begin{longtable}{lcl}
$\langle\langle x_{1},x_{2}\rangle,x_{3}\rangle$&$=$&
$p^2 x_1 x_2  x_3 - p q x_1 x_2  x_3 - p q x_1 x_3  x_2 + q^2 x_2 x_3  x_1,$\\
$\langle\langle x_{1},x_{3}\rangle,x_{2}\rangle$&$=$&
$p^2 x_1 x_3  x_2 - p q x_1 x_2  x_3 - p q x_1 x_3  x_2 + q^2 x_2 x_3x_1.$\end{longtable}

Therefore, 
\begin{longtable}{rcl}
$\langle\langle x_{1},x_{2}\rangle,x_{3}\rangle-\langle\langle x_{1},x_{3}\rangle,x_{2}\rangle$&$=$&$p^2 x_1 x_2  x_3-p^2 x_1 x_3  x_2,$\\
$ \langle\langle x_{2},x_{3}\rangle,x_{1}\rangle-\langle\langle x_{2},x_{1}\rangle,x_{3}\rangle$&$=$&$ p^2 x_2 x_3  x_1-p^2 x_1 x_2  x_3,$\\
$\langle\langle x_{3},x_{1}\rangle,x_{2}\rangle-\langle\langle x_{3},x_{2}\rangle,x_{1}\rangle$&$=$&$p^2 x_1 x_3  x_2-p^2 x_2 x_3  x_1.$
\end{longtable}

Substituting them into  $f(x_1,x_2,x_3)$, we see that $f(x_1,x_2,x_3)=0$ in $P_{p,q}$.

\end{proof}

Since every associative algebra satisfies $(\ref{liadm ident})$ and by Lemma $\ref{lem standard identity 3}$ we obtain $P_{p,q}$ satisfies the identity  \begin{equation}\label{standard identity 3 right normed}
    \tilde{f}(x_1,x_2,x_3):=\sum_{\sigma\in  {\mathbb S}_3}{sgn(\sigma)\langle x_{\sigma(1)},\langle x_{\sigma(2)}, x_{\sigma(3)}\rangle\rangle}=0.
\end{equation}

Set  

$$ \mathcal{WA}(a,b,c):=\langle a,b,c\rangle +\langle b,c,a\rangle-\langle b,a,c\rangle,$$
and 
$$\overline{\mathcal{WA}}(a,b,c):=\langle a,b,c\rangle +\langle b,c,a\rangle+\langle c,a,b\rangle.$$

In \cite{bremner-Juana} it was proved that the following identities hold in mutations of associative algebras ${\rm A}_{p,q}$:

$H(x_1,x_3,x_2,x_4)\ = $
\begin{flushright}
$= \ \overline{\mathcal{WA}}(x_1,x_3,x_2)x_4-\sum\limits_{\sigma\in \mathbb  S_3}\Big(\langle\langle x_{\sigma(1)},x_{\sigma(2)}\rangle,\langle x_{\sigma(3)},x_{4}\rangle\rangle-\langle x_{\sigma(1)},\langle\langle x_{\sigma(2)},x_{\sigma(3)}\rangle, x_{4}\rangle\rangle\Big)\ = \ 0,
$\end{flushright}

$
    I(x_1,x_3,x_2,x_4) \ =$
   \begin{flushright} $=\ \mathcal{WA}(\langle x_2,x_3\rangle,x_1,x_4)+\langle x_2,x_1\bullet x_4,x_3\rangle-\langle x_2,x_4,x_3\rangle \bullet x_1-\langle x_2,x_1,x_3\rangle \bullet x_4\ =\ 0,
$\end{flushright}
where $\langle x,y,z\rangle= \langle\langle x,y \rangle,z\rangle- \langle x,\langle y ,z\rangle \rangle$ and $x\bullet y = \langle x,y\rangle+\langle y,x\rangle.$

It is clear that the identities above hold in $P_{p,q}.$ But below we show that $ \mathcal{WA}(a,b,c)=0$ and $\overline{\mathcal{WA}}(a,b,c)=0$ in $P_{p,q}.$

\begin{lemma}\label{lem weakly associative identity}
Let $P$ be a $\mathfrak{perm}$ algebra. Then $P_{p,q}$ satisfies the identity \begin{equation}\label{weakly associative identity}
    \mathcal{WA}(x_1,x_2,x_3)\ :=\ 
    \langle x_1,x_2,x_3\rangle +\langle x_2,x_3,x_1\rangle-\langle x_2,x_1,x_3\rangle\ =\ 0.
\end{equation}
\end{lemma}

\begin{proof}  By Proposition $\ref{open brackets deg3}$ we have  
$$\begin{array}{lclcr}
\langle x_1,x_2,x_3\rangle & = & \langle \langle x_1,x_2\rangle, x_3\rangle-\langle x_1,\langle x_2,x_3\rangle\rangle & =&  -p q x_1x_2x_3  + p q x_2x_3x_1, \\
\langle x_2,x_3,x_1\rangle& =& \langle \langle x_2,x_3\rangle, x_1\rangle-\langle x_2,\langle x_3,x_1\rangle\rangle  & = & p q x_1x_3x_2  - p q x_2x_3x_1, \\ 
\langle x_2,x_1,x_3\rangle&= & \langle \langle x_2,x_1\rangle, x_3\rangle-\langle x_2,\langle x_1,x_3\rangle\rangle& = & -p q x_1x_2x_3  + p q x_1x_3x_2.\end{array}$$
Substituting them into  $\mathcal{WA}(x_1,x_2,x_3)$, we see that $\mathcal{WA}(x_1,x_2,x_3)=0$ in $P_{p,q}$.
\end{proof}

An algebra is called weakly associative if it satisfies the identity  $(\ref{weakly associative identity})$. Also, it was shown that weak associativity implies flexibility \cite[Proposition 3]{Remm}: \begin{equation}\label{flexibile-Id}
    \langle a,b,c\rangle+\langle c,b,a\rangle=0.
\end{equation}
Therefore, by \cite[Theorem 2.1]{book mutation} we have $P_{p,q}$ is a noncommutative Jordan algebra (about noncommutative Jordan algebras see \cite{aak24} and references therein). In particular, it is power-associative. Moreover, from \eqref{flexibile-Id} and $ \mathcal{WA}(a,b,c)=0$ implies $\overline{\mathcal{WA}}(a,b,c)=0$ in $P_{p,q}$
(identity $\overline{\mathcal{WA}}(a,b,c)=0$ defines dual cyclic associative algebras and it appears in \cite{aks24}).

\begin{corollary}
    Let $P$ be a $\mathfrak{perm}$ algebra. Then $P_{p,q}$ satisfy the identities \begin{longtable}{rcl}
     $\overline{H}(x_1,x_3,x_2,x_4)$&$=$&$\sum\limits_{\sigma\in \mathbb S_3}
     \Big(\langle\langle x_{\sigma(1)},x_{\sigma(2)}\rangle,\langle x_{\sigma(3)},x_{4}\rangle\rangle-\langle x_{\sigma(1)},\langle\langle x_{\sigma(2)},x_{\sigma(3)}\rangle, x_{4}\rangle\rangle\Big)\ =\ 0,$\\
     $\overline{I}(x_1,x_3,x_2,x_4)$&$=$&$\langle x_2,x_1\bullet x_4,x_3\rangle-\langle x_2,x_4,x_3\rangle \bullet x_1-\langle x_2,x_1,x_3\rangle \bullet x_4\ =\ 0.$\end{longtable}
\end{corollary}



\begin{proposition}\label{independence of the identites}
The identities $(\ref{standard identity 3})$ and $(\ref{weakly associative identity})$ are independent.
The identity $(\ref{standard identity 3 right normed})$ follows from $(\ref{standard identity 3})$ and $(\ref{weakly associative identity}).$   
\end{proposition}

\begin{proof} Let $\rm A$ be a three dimensional algebra with the multiplication table 
\begin{center}
    $\langle e_1,e_2\rangle=e_1, \  \langle e_2,e_1\rangle=-e_1,\ \langle e_3,e_1\rangle=e_2. $
\end{center} Then using Wolfram Mathematica code provided in \cite[Section 3]{KadMash2021} one can check that $\rm A$ satisfies the identity $(\ref{standard identity 3}),$ but $\mathcal{WA}(e_1,e_1,e_3)=-e_1\neq0.$ Therefore,  the identity $(\ref{standard identity 3})$ does not imply the identity $(\ref{weakly associative identity})$. Hence they are independent. For the second part of the statement we have
$$\tilde{f}(x,y,z)= \mathcal{WA}(z, y, x) -\mathcal{WA}(z, x, y) - f(z, y, x).$$
\end{proof}

\begin{theorem} \label{Th ident deg 3}
Over a field of  characteristic zero, every multilinear polynomial identity of degree $3$ satisfied by mutation of every $\mathfrak{perm}$ algebra is a consequence of the identities:
$$\begin{array}{rcl}
\langle a_1,a_2,a_3\rangle +\langle a_2,a_3,a_1\rangle -\langle a_2,a_1,a_3\rangle & = &  0, \\
 \sum\limits_{\sigma\in 
\mathbb S_3}{sgn(\sigma)\langle\langle a_{\sigma(1)},a_{\sigma(2)}\rangle,a_{\sigma(3)}\rangle} &  = & 0.
\end{array}$$
\end{theorem}
\begin{proof}  We have the following set of polynomials $(\ref{standard identity 3})$ and $(\ref {weakly associative identity})$ with all possible permutations in the variables $a,b, c:$
{\small \begin{equation} 
\Big\{\begin{array}{cccccc}\label{permut of ident deg3} 
\mathcal{WA}(a, b, c), & \mathcal{WA}(a, c, b), &\mathcal{WA}( b,a, c),&\mathcal{WA}(b, c, a),&\mathcal{WA}(c,a, b),&\mathcal{WA}(c, b, a),\\
f(a, b, c),&f(a, c, b),&f(b,a,c),&f(b,c,a),&f(c,a, b),&f(c, b, a)\}.\end{array} \Big\}\end{equation}}

There are 12 nonassociative monomials of degree three, and we present them in the following order:
\begin{equation} \begin{array}{cccccc}\label{order deg 3}
    a(bc),& a(cb),& b(ac),& b(ca),& c(ab),& c(ba),\\
    (ab)c,& (ac)b,& (ba)c,& (bc)a,& (ca)b,& (cb)a. 
\end{array}\end{equation}

We select the coefficients of the monomials in $(\ref{permut of ident deg3})$ relative to the order of $(\ref{order deg 3}).$ That is, the columns represent monomials in $ (\ref {order deg 3}), $ and the rows represent each polynomial in $(\ref{permut of ident deg3}).$ Then we have the following matrix:

$$\left(
\begin{array}{cccccccccccc}
 0 & 0 & 0 & 0 & 0 & 0 & 1 & -1 & -1 & 1 & 1 & -1 \\
 0 & 0 & 0 & 0 & 0 & 0 & -1 & 1 & 1 & -1 & -1 & 1 \\
 0 & 0 & 0 & 0 & 0 & 0 & -1 & 1 & 1 & -1 & -1 & 1 \\
 0 & 0 & 0 & 0 & 0 & 0 & 1 & -1 & -1 & 1 & 1 & -1 \\
 0 & 0 & 0 & 0 & 0 & 0 & 1 & -1 & -1 & 1 & 1 & -1 \\
 0 & 0 & 0 & 0 & 0 & 0 & -1 & 1 & 1 & -1 & -1 & 1 \\
 0 & 0 & 0 & -1 & -1 & 1 & 0 & 0 & 0 & 1 & 1 & -1 \\
 0 & 0 & -1 & 1 & 0 & -1 & 0 & 0 & 1 & -1 & 0 & 1 \\
 0 & -1 & 0 & 0 & 1 & -1 & 0 & 1 & 0 & 0 & -1 & 1 \\
 -1 & 1 & 0 & 0 & -1 & 0 & 1 & -1 & 0 & 0 & 1 & 0 \\
 -1 & 0 & 1 & -1 & 0 & 0 & 1 & 0 & -1 & 1 & 0 & 0 \\
 1 & -1 & -1 & 0 & 0 & 0 & -1 & 1 & 1 & 0 & 0 & 0 \\
\end{array}
\right).$$

The rank of the above matrix is $5.$ Hence, the multilinear dimension of the algebra defined by the identities $(\ref{standard identity 3})$ and $(\ref{weakly associative identity})$ in the third degree is equal to $7.$ Suppose there is an identity in $P_{p,q}(X) $ that is not a consequence of $(\ref{standard identity 3})$ and $ (\ref {weakly associative identity}).$ Then the dimension of $P_{p,q}(X) $ should be less than $7, $ but by the Theorem $\ref{mut element and dim}$ this is impossible. Hence, every identity of the degree three follows from $(\ref{standard identity 3})$ and $(\ref {weakly associative identity}).$

\end{proof}

Let us define the commutator product of two elements $x$ and $y$ in $P_{p,q}(X)$  by \begin{center}$x\circ y \ :=\ \langle x, y\rangle - \langle y,x \rangle \ =\ (p+q)[x,y].$ \end{center} Then it is easy to see that the following identity holds in $P_{p,q}(X)$
    $$\langle a\circ b, c\circ d\rangle\ =\ 0.$$
Therefore, $P_{p,q}(X)$ under commutator product $x\circ y$ is a metabelian Lie algebra.

In \cite{Boers},   Boers showed that there exists a Lie-admissible algebra whose mutation algebra is not Lie-admissible. In addition, he gave an example of a non-Lie-admissible algebra, but its mutation algebra is Lie-admissible. The following theorem gives a necessary and sufficient condition for algebras whose mutation algebras are Lie-admissible.

\begin{theorem}
   Let $\rm A$ be a nonassociative algebra over a field of characteristics zero. Then the mutation algebra ${\rm A}_{p,q}$  is Lie-admissible, for every nonzero $p,q\in {\rm A},$ if and only if the following identity holds in ${\rm A}:$

   \begin{equation}\label{idforLiadm}
       \sum_{\sigma\in {\mathbb S}_3}sgn(\sigma)\big((x_{\sigma(1)} y) ((x_{\sigma(2)} y) x_{\sigma(3)})-(((x_{\sigma(1)} y) x_{\sigma(2)}) y) x_{\sigma(3)}\big) =0.\end{equation}
\end{theorem}
\begin{proof}
Let ${\rm A}$  be a nonassociative algebra with identity \eqref{idforLiadm}. Then by  the definitions of $x\circ y :=\langle x, y\rangle - \langle y,x \rangle$ and the mutation product $\langle x, y\rangle=(x p)y-(yq)x,$ we have

$$(a\circ b)\circ c+(b\circ c)\circ a+(c\circ a)\circ b=$$
$$-((a p) ((b p) c))  + ((a p) ((c p) b))
+((b p) ((a p) c))  - ((b p) ((c p) a))
-((c p) ((a p) b))  + ((c p) ((b p) a))
$$$$+((((a p) b) p) c)  - ((((a p) c) p) b)
-((((b p) a) p) c)  + ((((b p) c) p) a)
+((((c p) a) p) b)  - ((((c p) b) p) a)$$
$$
- ((a q) ((b q) c))  + ((a q) ((c q) b))
+ ((b q) ((a q) c))  - ((b q) ((c q) a))   
- ((c q) ((a q) b))  + ((c q) ((b q) a))  $$$$ 
+ ((((a q) b) q) c)  - ((((a q) c) q) b)   
- ((((b q) a) q) c)  + ((((b q) c) q) a)   
+ ((((c q) a) q) b)  - ((((c q) b) q) a)$$$$
- ((a p) ((b q) c))+ ((a p) ((c q) b))-((a q) ((b p) c)) + ((a q) ((c p) b))
+ ((b p) ((a q) c))- ((b p) ((c q) a))$$$$-((b q) ((c p) a)) - ((c p) ((a q) b))
+ ((b q) ((a p) c))+ ((c p) ((b q) a))-((c q) ((a p) b)) + ((c q) ((b p) a))$$$$
+ ((((a p) b) q) c)- ((((a p) c) q) b)+((((a q) b) p) c) - ((((b p) a) q) c)
- ((((a q) c) p) b)+ ((((b p) c) q) a)$$$$-((((b q) a) p) c) + ((((b q) c) p) a)
+ ((((c p) a) q) b)- ((((c p) b) q) a)+((((c q) a) p) b) - ((((c q) b) p) a).$$
Hence, by the identity \eqref{idforLiadm} and by its full linearization we obtain
$$(a\circ b)\circ c+(b\circ c)\circ a+(c\circ a)\circ b=0.$$ The converse follows from the homogeneous property of identities. 
\end{proof}

An algebra with identities
\begin{center}
    $(ab)c=(ac)b$ \  and \  $a(bc)=b(ac)$  
\end{center}
is called a bicommutative algebra.  

\begin{corollary}
    Let $\rm B$ be a bicommutative algebra. Then ${\rm B}_{p,q}$ is a Lie-admissible algebra.
\end{corollary}

\section{Homomorphic image of mutations of $\mathfrak{perm}$ algebras}

In this section, we show that there exists an exceptional homomorphic image of $P_{p,q}(X).$ 
We use a general method that includes non-special homomorphic images of special Jordan algebras to construct counterexamples, illustrated by P.M. Cohn in \cite{Cohn}. 
The concept of an associative algebra and the special Jordan algebras embedded within it can be generalised to the case of an arbitrary algebra $\rm A$ and its subsets, which are closed under certain combinations of the operators of $\rm A.$

\begin{lemma}[Cohn, \cite{Cohn}]\label{crit hom image}
    A homomorphic image $A_{p,q}(X)/I$ of the free special mutation algebra ${\rm A}_{p,q}(X) $ is special if and only if $ \tilde{I} \cap {\rm A}_{p,q}(X) = I,$ where $I$ is an ideal of  ${\rm A}_{p,q}(X)$ and  $\tilde{I}$ is the ideal of the free algebra  ${\rm A}(X\cup \{p,q\})$ generated by $I.$
\end{lemma}

   The proof follows the same reasoning presented in  \cite{Cohn} or see \cite[Ch. 3]{ZSSS}.

The above lemma is crucial in proving the existence of exceptional homomorphic images. This method has also been used to obtain exceptional homomorphic images of various classes of algebras; for example, see \cite{dim19, FM-BS}.

Now, we prove the following theorem:

\begin{theorem}\label{hom image} There exists a homomorphic image of $P_{p,q}(\{x_1,x_2,x_3,x_4\})$ which is  exceptional.
\end{theorem}
\begin{proof}
Let $I$ be an ideal of $P_{p,q}(\{x_1,x_2,x_3,x_4\})$ mutation of $\mathfrak{perm}$ algebra $P(\{x_1,x_2,x_3,x_4,p,q\})$ generated by  $f_1=\langle\langle x_2,x_3\rangle,x_4 \rangle$ and $f_2=\langle\langle x_2,x_3\rangle, x_1 \rangle.$
Define $\tilde{I}$ ideal of $P(\{x_1,x_2,x_3,x_4,p,q\})$ generated by the set $I.$ Set $a=x_1,$ $b=\langle x_2,x_3\rangle,$ and  $c=x_4$ in $(\ref{hom ident1}),$ and we have 
\begin{equation}\label{hom ident2}
    \langle\langle x_2,x_3\rangle,\langle x_1,x_4 \rangle \rangle - x_1  p  \langle\langle x_2,x_3\rangle,x_4 \rangle+ x_4  q \langle\langle x_2,x_3\rangle, x_1 \rangle =0.
\end{equation}

It is clear that from $(\ref{hom ident2})$ follows $ \langle\langle x_2,x_3\rangle,\langle x_1,x_4 \rangle \rangle\in \tilde{I}\cap P_{p,q}(\{x_1,x_2,x_3,x_4\}).$

Let us to show that $ \langle\langle x_2,x_3\rangle,\langle x_1,x_4 \rangle \rangle\notin I,$ i.e. there are no $\lambda_1,\lambda_2,\lambda_3,\lambda_4$ such that $$\begin{array}{lcl} 
\langle\langle x_2,x_3\rangle,\langle x_1,x_4 \rangle \rangle & = & \lambda_1\langle \langle\langle x_2,x_3\rangle,x_4 \rangle,x_1\rangle+\lambda_2\langle x_1,\langle\langle x_2,x_3\rangle,x_4 \rangle\rangle +\\
&& \lambda_3\langle \langle\langle x_2,x_3\rangle, x_1 \rangle,x_4\rangle+ \lambda_4\langle x_4, \langle\langle x_2,x_3\rangle, x_1 \rangle\rangle. 
\end{array}$$
\bigskip 

We expand the mutation brackets $\langle x,y\rangle=xpy-yqx$ into the free $\mathfrak{perm}$ algebra $P(\{x_1,x_2,x_3,x_4,p,q\})$ and collect the coefficients of base elements of $P(\{x_1,x_2,x_3,x_4,p,q\}).$ 
Then we have $12$ linear equations with $4$ unknowns:

\begin{longtable}{rclrcl}
$\lambda _2+\lambda _3$&$=$&$1,$ &
     $\lambda _1+\lambda _2+2 \lambda _3+\lambda _4$&$=$&$1,$\\
    
     $\lambda _4$&$=$&$0,$ & $\lambda _1+\lambda _3+2 \lambda _4$&$=$&$0,$  \\    
     $\lambda _2+\lambda _4$&$=$&$1,$ &
     $\lambda _1+\lambda _3$&$=$&$1,$\\ 
     $\lambda _2+\lambda _4$&$=$&$1,$ &
     $\lambda _1+\lambda _3$&$=$&$1,$ \\
     $\lambda _1+\lambda _4$&$=$&$0, $ &
     $
     2 \lambda _1+\lambda _2+\lambda _3+\lambda _4$&$=$&$1,$ \\ 
     $\lambda _1+2 \lambda _2+\lambda _3$&$=$&$1,$ &
     $\lambda _2$&$=$&$0$

\end{longtable}

The above system of linear equations does not have a solution. Therefore, $ \langle\langle x_2,x_3\rangle,\langle x_1,x_4 \rangle \rangle\notin I.$ By Cohn's criterion Lemma \ref{crit hom image} the homomorphic image  $P_{p,q}(\{x_1,x_2,x_3,x_4\})/I$ is exceptional.
\end{proof}

\section{Open question}

The method for finding identities using computer algebra is described in detail in \cite{bremner-Juana}. Using this method, one can also show that, in addition to identities 
\begin{center}
    $f(a,b,c)=0,$\ 
    $\mathcal{WA}(a,b,c)=0,$  \  
    $\overline{H}(a,b,c,d)=0,$  \ 
    $\overline{I}(a,b,c,d)=0,$
\end{center}
there are two more new identities at degree four. It is noteworthy that, although in the case of associative algebras, a new identity appears at each degree up to the $7$th degree, but in the case of $\mathfrak{perm}$ algebras, all identities follow from these six identities up to degree $7$. This observation leads the following:

\begin{conjecture}
    
 Every multilinear identity in the mutation of a \(\mathfrak{perm}\) algebra over a field of characteristic zero is a consequence of the following identities:
\[
f(a,b,c) = 0, \quad \mathcal{WA}(a,b,c) = 0, \quad \overline{H}(a,b,c,d) = 0, \quad \overline{I}(a,b,c,d) = 0,
\]
and two additional identities in degree four:
  \relax 

{\small  
\begin{longtable}{rcl}
 
$\langle \langle \langle a,b \rangle,  c \rangle,  d \rangle  +\langle \langle \langle c,d \rangle,  a \rangle,  b \rangle$&$=$&$\langle \langle \langle a,d \rangle,  c \rangle,  b \rangle  +\langle \langle \langle c,b \rangle,  a \rangle,  d \rangle ,$  \\

 $\langle \langle a,b \rangle,  \langle d,c \rangle   \rangle  
 +\langle \langle c,\langle b,a \rangle   \rangle,  d \rangle  
 +\langle \langle \langle b,a \rangle,  c \rangle,  d \rangle  
 $&$=$&$
 \langle \langle \langle a,b \rangle,  d \rangle,  c \rangle  
 +\langle \langle \langle b,c \rangle,  a \rangle,  d \rangle  
 +\langle \langle \langle c,a \rangle,  b \rangle,  d \rangle.$
   \end{longtable}}
\end{conjecture}
Theorem \ref{mut element and dim} provides a lower bound for the dimension of the algebra defined by the above identities. The remaining task is to find a basis for the free algebra defined by the identities in the aforementioned conjecture.

\end{document}